\documentclass{amsart}
\usepackage{hyperref}
\usepackage{color}
\definecolor{mylinkcolor}{rgb}{0.5,0.0,0.0}
\definecolor{myurlcolor}{rgb}{0.0,0.0,0.75}
\hypersetup{colorlinks=true,urlcolor=myurlcolor,citecolor=myurlcolor,linkcolor=mylinkcolor,linktoc=page,breaklinks=true}
\usepackage{amsmath,mathtools,bbm}
\usepackage{booktabs,array,enumerate}
\usepackage{thmtools,thm-restate}
\declaretheorem{theorem}
\usepackage{fullpage}
\usepackage[charter]{mathdesign}

\usepackage[section]{placeins}
\usepackage{siunitx}
\usepackage{qtree}
\usepackage{hyperref}

\renewcommand{\O}{\mathcal{O}}

\newcommand{\Cartier}{\mathcal{C}}
\newcommand{\C}{\mathbf{C}}
\newcommand{\Q}{\mathbf{Q}}
\newcommand{\Z}{\mathbf{Z}}

\newcommand{\F}{\mathbf{F}}
\newcommand{\Fp}{\F_p}
\DeclareMathOperator{\M}{\mathsf{M}}

\newcommand{\tr}{\operatorname{tr}}
\newcommand{\lc}{\operatorname{lc}}

\newcommand{\rem}{\operatorname{rem}}

\renewcommand{\vec}[1]{{\boldsymbol{#1}}}
\renewcommand{\div}{\operatorname{div}}

\newcommand{\defi}{\textsf}
\newcommand{\disc}{\operatorname{disc}}
\newcommand{\mr}[1]{MathSciNet:\,\href{https://mathscinet.ams.org/mathscinet-getitem?mr=#1}{MR#1}}
\newcommand{\arxiv}[1]{arXiv:\,\href{https://arxiv.org/abs/#1}{#1}}
\newcommand{\halinria}[1]{HAL-Inria:\,\href{https://hal.inria.fr/inria-#1/}{#1}}
\newcommand{\hal}[1]{HAL:\,\href{https://hal.archives-ouvertes.fr/hal-#1/}{#1}}

\newtheorem{lemma}[theorem]{Lemma}
\newtheorem{corollary}[theorem]{Corollary}

\theoremstyle{definition}
\newtheorem{definition}[theorem]{Definition}

\newtheorem{example}[theorem]{Example}
\newtheorem{remark}[theorem]{Remark}

\title{Counting points on superelliptic curves in average polynomial time}
\author{Andrew V. Sutherland}
\thanks{The author was supported by Simons Foundation grant 550033}

\begin{document}

\begin{abstract}
We describe the practical implementation of an average polynomial-time algorithm for counting points on superelliptic curves defined over $\Q$ that is substantially faster than previous approaches.  Our algorithm takes as input a superelliptic curve $y^m=f(x)$ with $m\ge 2$ and $f\in \Z[x]$ any squarefree polynomial of degree $d\ge 3$, along with a positive integer $N$.  It can compute $\#X(\Fp)$ for all $p\le N$ not dividing $m\lc(f)\disc(f)$ in time $O(md^3 N\log^3 N\log\log N)$.  It achieves this by computing the trace of the Cartier--Manin matrix of reductions of $X$.  We can also compute the Cartier--Manin matrix itself, which determines the $p$-rank of the Jacobian of $X$ and the numerator of its zeta function modulo~$p$.
\end{abstract}

\maketitle
{\centering\small\textit{In memory of \href{https://en.wikipedia.org/wiki/Peter_Montgomery_(mathematician)}{Peter L. Montgomery}.}\par}

\section{Introduction}
Let $X/k$ by a smooth projective curve of genus $g>0$ whose function field is defined by an equation of the form
\[
y^m = f(x),
\]
with $m>1$ prime to the characteristic $p$ of $k$ and $f\in k[x]$ a squarefree polynomial of degree $d\ge 3$.
We shall call such a curve $X$ a \defi{superelliptic curve}.  We note that not all authors require $f$ to be squarefree or $p\nmid m$, while others require $d$ and $m$ to be coprime; our definition follows the convention in \cite{MN17,Z18} and is equivalent to the class of cyclic covers of $\mathbf{P}^1$ considered in \cite{ABCMT19,Gon15}.
One can compute the genus of $X$ as
\begin{equation}\label{eq:g}
g = \frac{(d-2)(m-1) + m-\gcd(m,d)}{2},
\end{equation}
via the Riemann-Hurwitz formula.
Well known examples of superelliptic curves include elliptic curves, hyperelliptic curves, Picard curves, and Fermat curves.

We are primarily interested in $k=\Q$ where $X$ has an associated $L$-function $L(X,s)=\sum a_nn^{-s}$ that we would like to ``compute''.  For us this means computing the integers $a_n$ for all $n$ up to a bound $N$ that is large enough for us to approximate special values of $L(X,s)$ to high precision, and to compute upper bounds on its analytic rank that we can reasonably expect to be sharp.  This requires $N$ to be on the order of the square root of the conductor of the Jacobian of $X$, and in practice we typically take $N$ to be about 30 times this value.

The fact that $L(X,s)$ is defined by an Euler product implies that it suffices to compute $a_n$ for prime powers $n\le N$.  Nearly all of the prime powers $n\le N$ are in fact primes $p$, so this task is overwhelmingly dominated by the time to compute $a_p$ for primes $p\le N$.  Indeed, if we spend $O(p^{e-1}\log^2 p)$ time computing each $a_{p^e}\le N$ with $e > 1$, we will have spent only $O(N\log N)$ time, which is roughly the time it takes just to write down the $a_n$ for $n\le N$.  For primes of good reduction for $X$, which includes all $p\nmid m\lc(f)\disc(f)$,\footnote{When $m$ divides $d$ there may be good primes that divide $\lc(f)$, but to simplify the presentation we shall exclude them.} we can compute $a_p$ as
\[
a_p = p+1-\#X(\Fp),
\]
in other words, by counting points on the reduction of $X$ modulo $p$.  See \cite{BW17} for a discussion of how primes of bad reduction may be treated.  Alternatively, if one is willing to assume that the Hasse-Weil conjecture for $L(X,s)$ holds, one can use the knowledge of $a_n$ at powers of good primes to determine the $a_n$ at powers of bad primes (and in particular, the primes $p|m$ not treated by \cite{BW17}) by using the functional equation to rule out all but one possibility; see \cite[\S 5]{BSSVY16} for a discussion of this approach when $g=2$.  

Another motivation for computing $a_p$ for good primes $p\le N$ is to compute the sequence of normalized Frobenius traces $a_p/\sqrt{p}$ that appear in generalizations of the  Sato-Tate conjecture.  The moments of this distribution  encode certain arithmetic invariants of $X$, including, for example, the rank of the endomorphism ring of its Jacobian \cite[Prop,\,1]{CFS19}, as well as information about its Sato-Tate group \cite{Ser12,FKRS12}.  Indeed, the initial motivation for this work (and its first application) was to compute Sato-Tate distributions for the three types of genus 3 superelliptic curves with $(m,d)\in \{(3,4),(4,3),(4,4)\}$ that arise as smooth plane quartics in the database described in \cite{Sut18}, which played a role in the recent classification of Sato--Tate groups of abelian threefolds \cite{FKS19}. The sequence of normalized Frobenius traces can also be used to numerically investigate the  error term in the Sato-Tate conjecture, and in particular, predictions regarding its leading constant \cite{BKF20}. The ability to efficiently compute many integer values of $a_p$ also supports investigations of generalizations of the Lang-Trotter conjecture, as well as a recent question of Serre regarding the density of ``record'' primes, those for which $-a_p > 2g\sqrt{p}-1$ \cite{Ser19}.

The algorithm we present here actually does more than just compute $a_p$.  Following the approach of \cite{HMS16, HS14, HS16}, which treated the case of hyperelliptic curves, for each good prime $p$ we compute a $g\times g$ matrix $A_p$ giving the action of the Cartier--Manin operator on a basis for the space of regular differentials of the reduction of $X$ modulo~$p$; see~\S\ref{sec:cartier} for details.  This matrix $A_p$ is the transpose of the Hasse--Witt matrix, and like the Hasse--Witt matrix it satisfies the identity
\[
\det (I-TA_p)\equiv L_p(T)\bmod p,
\]
where $L_p(T)$ is the integer polynomial that appears in both the Euler product $L(X,s)=\prod_p L_p(p^{-s})^{-1}$ and the numerator of the zeta function of the reduction of $X$ modulo $p$:
\[
Z_p(T) := \exp\left(\sum_{n\ge 1} \#X(\F_{p^n})\frac{T^n}{n}\right) =\frac{L_p(T)}{(1-T)(1-pT)}.
\]
In particular, we have $a_p\equiv \tr A_p\bmod p$, and for $p>16g^2$ this uniquely determines $a_p\in \Z$, since $|a_p|\le 2g\sqrt{p}$, by the Weil bounds.  The matrix $A_p$ is also of independent interest, since it can be used to compute the $p$-rank of the reduction of $X$ modulo $p$, something that cannot be deduced solely from $L_p(T)\bmod p$.

Our main result is the following theorem, in which $\|f\|\coloneqq \log\max_i |f_i|$ denotes the logarithmic height of a nonzero integer polynomial $f(x)=\sum_i f_ix^i$.

\begin{restatable}{theorem}{firstavgpoly}\label{thm:avgpoly}
Given a superelliptic curve $X\colon y^m=f(x)$ with $f\in \Z[x]$ of degree~$d$ and $N\in \Z_{> 0}$, the algorithm \textsc{ComputeCartierManinMatrices} outputs the Cartier--Manin matrices $A_p$ of the reductions of $X$ modulo all primes $p\le N$ not dividing $m\lc(f)\disc(f)$.  If we assume $m$, $d$, $\|f\|$ are bounded by $O(\log N)$ the algorithm runs in $O(m^2d^3N \log^3\!N)$ time using $O(md^2N)$ space;
it can alternatively compute Frobenius traces $a_p\in \Z$ for $p\le N$ in time $O(md^3N \log^3\!N)$.
\end{restatable}

\begin{remark}
The assumption $m,d,\|f\|=O(\log N)$ ensures that the complexity of multiplying the integer matrices used in the algorithm is dominated by the cost of computing FFT transforms of the matrix entries, which eliminates any dependence on the exponent $\omega$ of matrix multiplication; one can replace $d^3$ with $d^{\omega+1}$ and then remove this assumption.  We note that our complexity bound relies on the recently improved $\M(n)=n\log n$ bound on integer multiplication \cite{HvdH19a}.  While the algorithm that achieves this bound is not practical, many FFT-based implementations effectively achieve this growth rate within the feasible range of computation, which for our purposes, is certainly limited to integers that fit in random access memory; see \cite[Alg.\,8.25]{GG13}, for example.
\end{remark}

We also obtain an algorithm that can be used to compute $A_p$ for a single superelliptic curve $X/\Fp$.  The asymptotic complexity is comparable to that achieved in \cite{ABCMT19} which describes the algorithm that is now implemented in version 9 of Sage \cite{sage}.  We include this result because it contains several components that are used by the average polynomial-time algorithm we present.  We should emphasize that the algorithm in \cite{ABCMT19} can compute $L_p(T)\bmod p^n$ for any $n\ge 1$, and taking $n$ sufficiently large yields $L_p\in \Z[T]$, whereas we focus solely on the case $n=1$ (we gain a small but not particularly significant performance advantage in this case).

\begin{restatable}{theorem}{firstmodp}\label{thm:modp}
Given a superelliptic curve $X\colon y^m=f(x)$ with $f\in \Fp[x]$ of degree $d$, the algorithm \textsc{ComputeCartierManinMatrix} can compute the Cartier--Manin matrix of $X$ in $O(md^3 p^{1/2}\log p (d^{\omega-2}\log\log p+\log p))$ time using $O(md^2p^{1/2}\!\log p)$ space, and also in $O(md^2(p+d)\log p\log\log p)$ time using $O((md+d^2)\log p)$ space.
\end{restatable}

In the  article \cite{ABCMT19} noted above the authors consider a particular curve
\[
X\colon y^7 = x^3 + 4x^2 + 3x - 1,
\]
for which they estimate that it would take approximately six months (on a single core) for their algorithm to compute the $L$-polynomials $L_p(T)$ for all primes $p\le 2^{24}$ of good reduction.  This is an improvement over an estimated three years for an earlier algorithm due to Minzlaff \cite{Minz10} that is implemented in Magma~\cite{magma}.  Computing $L_p(T)\bmod p$ is an easier problem that would likely take about a week or so using the algorithm in \cite{ABCMT19}, based on timings taken using a representative sample of $p\le 2^{24}$.  The algorithm we present here can accomplish this task in half an hour, and less than ten minutes if we only compute Frobenius traces.

See Tables~\ref{tab:perfhard} and~\ref{tab:perfeasy} in \S\ref{sec:perf} for detailed performance comparisons for various shapes of superelliptic curves.

\section{The Cartier operator}\label{sec:cartier}

For background on differentials of algebraic function fields we refer the reader to \cite[\S 2]{Ch51} and \cite[\S 4]{St09}.
Let $K$ be a function field of one variable over a perfect field $k$ of characteristic $p>0$ that we assume is the full field of constants of $K$.
Let $\Omega_K$ denote its module of differentials, which we identify with its module of Weil differentials via \cite[Def.\,4.17]{St09} and \cite[Rm.\,4.3.7]{St09}.
Let $x\in K$ be a separating element, so that $K/k(x)$ is a finite separable extension, and let $K^p$ denote the subfield of $p$th powers.  Then $(1,x,\ldots, x^{p-1})$ is a basis for $K$ as a $K^p$-vector space, and every $z\in K$ has a unique representation of the form
\[
z = z_0^p+z_1^px+\cdots + z_{p-1}^px^{p-1},
\]
with $z_0,\ldots,z_{p-1}\in K$, and every rational differential form $\omega = zdx$ can be uniquely written in the form
\[
\omega = (z_0^p+z_1^px+\cdots z_{p-1}^px^{p-1})dx.
\]
The (modified) \emph{Cartier operator} $\Cartier\colon \Omega_K\to\Omega_K$ is then defined by
\[
\Cartier(\omega) \coloneqq z_{p-1}dx.
\]
The Cartier operator is uniquely characterized by the following properties:
\begin{enumerate}
\item $\Cartier(\omega_1+\omega_2)=\Cartier(\omega_1)+\Cartier(\omega_2)$ for all $\omega_1,\omega_2\in\Omega_K$;
\item $\Cartier(z^p\omega)=z\,\Cartier(\omega)$ for all $z\in K$ and $\omega\in \Omega_K$;
\item $\Cartier(dz) = 0$ for all $z\in K$;
\item $\Cartier(dz/z)=dz/z$ for all $z\in K^\times$.
\end{enumerate}
In particular, it does not depend on our choice of a separating element $x$.  Moreover, it maps regular differentials to regular differentials and thus restricts to an operator on the space $\Omega_K(0)\coloneqq\{\omega\in\Omega_K:\omega=0\text{ or } \div(\omega)\ge 0\}$, which we recall is a $k$-vector space whose dimension $g$ is equal to (and often used as the definition of) the genus of $K$; see \cite[Ex.\,4.12-17]{St09} for these and other standard facts about the Cartier operator.

\begin{definition}
Let $\vec{\omega}\coloneqq (\omega_1,\ldots,\omega_g)$ be a basis for $\Omega_K(0)$ and define $a_{ij}\in k$ via
\[
\Cartier(\omega_j)=\sum_{i=1}^g a_{ij}\omega_i.
\]
The \defi{Cartier--Manin} matrix of $K$ (with respect to $\vec{\omega}$) is the matrix $A\coloneqq [a_{ij}]\in k^{g\times g}$.
\end{definition}

If $X/k$ is a smooth projective curve with function field $k(X)=K$, we also call $A$ the Cartier--Manin matrix of $X$.  This matrix is closely related to the \defi{Hasse-Witt} matrix $B$ of $X$, which is defined as the matrix of the $p$-power Frobenius operator acting on $H^1(X,\O_X)$ with respect to some basis.  As carefully explained in \cite{AH19}, the matrices $A$ and $B$ can be related via Serre duality, and for a suitable choice of basis one finds that $B = [a_{ij}^p]^{\mathsf{T}}$.
In the case of interest to us $k=\Fp$ is a prime field and the Cartier--Manin and Hasse--Witt matrices are simply transposes of eachother, hence have the same rank and characteristic polynomials, but we shall follow the warning/request of~\cite{AH19} and call $A$ the Cartier--Manin matrix, although one can find examples in the literature where $A$ is called the Hasse--Witt matrix (see \cite{AH19} for a list).

We shall apply the method of St\"ohr--Voloch \cite{SV87} to compute the Cartier--Manin matrix of a smooth projective curve $X$ with function field $K=k(X)$.
Let us write $K$ as $k(x)[y]/(F)$, where $x\in X$ is a separating element and $y$ is an integral generator for the finite separable extension $K/k(x)$ with minimal polynomial $F\in k[x][y]$.
We now define the differential operator
\[
\nabla \coloneqq \frac{\partial^{2p-2}}{\partial x^{p-1}\partial y^{p-1}},
\]
which maps $x^{(i+1)p-1}y^{(j+1)p-1}$ to $x^{ip}y^{jp}$ and annihilates monomials not of this form; it thus defines a semilinear map $\nabla\colon K\to K^p$.
Writing $F_y$ for $\frac{\partial}{\partial y} F\in k[x,y]$, for any $h\in K$ we have the identity
\begin{equation}\label{eq:SV1}
\Cartier\left(h\frac{dx}{F_y}\right) = \left(\nabla (F^{p-1}h)\right)^{1/p}\frac{dx}{F_y}
\end{equation}
given by \cite[Thm.\,1.1]{SV87}.
If we choose a basis for $\Omega_X(0)$ using regular differentials of the form $h dx/F_y$, we can compute the action of the Cartier operator on this basis via \eqref{eq:SV1}.  To construct such a basis we shall use differentials of the form
\begin{equation}\label{eq:omega}
\omega_{k\ell}\coloneqq x^{k-1}y^{\ell-1}\frac{dx}{F_y}\qquad (k,\ell\ge 1,\ \ k+\ell\le \deg(F)-1).
\end{equation}
Writing $F(x,y)^{p-1}=\sum_{i,j} F^{p-1}_{ij}x^iy^j$ (defining $F^{p-1}_{i,j}\in k$ for all $i,j\in\Z$), for $k,\ell\ge 1$ one finds that
\begin{equation}\label{eq:SV2}
\nabla\left(\sum_{i,j\ge 0}F^{p-1}_{ij}x^{i+k-1}y^{j+\ell-1}\right) = \sum_{i,j\ge 1} F^{p-1}_{ip-k,\,jp-\ell} x^{(i-1)p}y^{(j-1)p}.
\end{equation}
Now $F^{p-1}_{ip-k,\,jp-\ell}$ is nonzero only if we have $(i+j)p-(k+\ell)\le (p-1)\deg(F)$, and $k+\ell\le \deg(F)-1$, so we can restrict the sum on the RHS to $i+j\le \deg(F)-1$.
From \eqref{eq:SV1} and \eqref{eq:SV2} we obtain
\begin{equation}\label{eq:SV3}
\Cartier(\omega_{k\ell}) = \sum_{i,j\ge 1} \left(F_{ip-k,\,jp-\ell}^{p-1}\right)^{1/p}\omega_{ij}.
\end{equation}
When $X$ is a smooth plane curve the complete set of $\omega_{ij}$ defined in \eqref{eq:omega} is a basis for $\Omega_K(0)$ and we can read off the entries of the Cartier--Manin matrix for $X$ directly from \eqref{eq:SV3}.
In general not all of the $\omega_{ij}$ necessarily lie in $\Omega_K(0)$, some of them might not be regular, but the subset that do (those corresponding to \defi{adjoint polynomials}) form a basis for $\Omega_K(0)$; see \cite{Go52,SV87}.
In the case of superelliptic curves this subset is given explicitly by Lemma~\ref{lem:basis} below.

\begin{definition} For $a,b\in\Z$ with $b>0$ let $a\rem b\coloneqq a - b\lfloor a/b\rfloor$ denote the unique integer in $[0,b-1]\cap (a+b\Z)$.
\end{definition}

\begin{lemma}\label{lem:basis}
Let $k$ be a perfect field of positive characteristic $p$, let $X/k$ be a superelliptic curve defined by $F(x,y)\coloneqq y^m-f(x)=0$, let $d\coloneqq \deg f$, and for $i,j\ge 1$ let $\omega_{ij}\coloneqq x^{i-1}y^{j-1} dx/F_y\in \Omega_K$, where $K\coloneqq k(x)[y]/(F)$ is the function field of $X$.  Then the set
\[
\vec \omega \coloneqq \{ \omega_{ij}\colon mi + dj  < md \},
\]
is a $k$-basis for $\Omega_K(0)$, with $1\le i < d-\lfloor d/m\rfloor$ and $1\le j < m-\lfloor m/d\rfloor$.
Moreover, if we define
\begin{equation}\label{eq:dj}
d_j\coloneqq d - \lfloor dj/m\rfloor - 1\qquad\text{and}\qquad m_i\coloneqq m-\lfloor mi/d\rfloor -1,
\end{equation}
then the $\omega_{ij}\in\vec{\omega}$ are precisely those for which $1\le i\le d_j$ and $1\le j \le m_i$.
\end{lemma}
\begin{proof}
Note that $\omega_{ij}=\frac{1}{m}x^{i-1}y^{j-m}dx$,
with $p\nmid m$.
It follows from \cite[3.8]{MN17} (which treats $X/\C$ but whose proof also works for $X/k$ and can be independently derived using the methods of \cite{Go52}) that the set
\[
\{x^{i-1}y^{-k}dx: 1\le i < d,\ 1\le k\le m-1,\ dk-mi \ge \gcd(m,d)\}
\]
is a basis for $\Omega_K(0)$. Taking $k=m-j$ and rearranging yields the basis
\[
\vec\omega = \{\omega_{ij}: mi+dj\le md-\gcd(m,d)\} = \{\omega_{ij}:mi+dj<md\},
\]
and the bounds on $i$ and $j$ immediately follow.
\end{proof}
For $X/k$ defined by $F(x,y)=f(x)-y^m=0$, if we let $f^n_a$ denote the coefficient of $x^a$ in $f(x)^n$ then
\[
F^{p-1}_{a,b}= \begin{cases}f^{p-1-b/m}_a, &\text{if } m\mid b\text{ and }b \le m(p-1), \\ 0&\text{otherwise},\end{cases}
\]
(here we have used $\binom{p-1}{e}(-1)^e\equiv 1\bmod p$), thus for all $1\le i,k < d$ and $1\le j,\ell < m$ we have
\[
F^{p-1}_{ip-k,\,jp-\ell} = \begin{cases} f^{p-1 - (jp-\ell)/m}_{ip-k}&\text{if } m\mid (jp-\ell), \\ 0&\text{otherwise}.\end{cases}
\]
Now $1\le j,\ell < m$ and $p\nmid m$, so whenever $F^{p-1}_{ip-k,\,jp-\ell}\ne 0$ we must have $\ell=jp\rem m > 0$ and
\begin{equation}\label{eq:nj}
n_j \coloneqq p-1-(jp-\ell)/m = \frac{(m-j)p-(m-\ell)}{m} = p-1-\lfloor jp/m\rfloor.
\end{equation}

Let us order the basis for $\Omega_K(0)$ given by Lemma~\ref{lem:basis} as $\vec{\omega} = (\omega_{11},\omega_{21},\ldots,\omega_{12},\ldots)$ with the $\omega_{ij}$ ordered first by~$j$ and then by~$i$.  The Cartier--Manin matrix of $X$ can then be described in block form with blocks indexed by $j$ and $\ell$ containing entries indexed by $i$ and $k$:
\begin{align}\label{eq:B}
A_p &\coloneqq  [B^{j\ell}]_{j\ell}\qquad\qquad\qquad\, 1\le j,\ell \le \mu\coloneqq m_1 = m - \lfloor m/d \rfloor - 1,\\\notag
B^{j\ell} &\coloneqq [(b^{j\ell}_{ik})^{1/p}]_{ik}\,\qquad\qquad 1\le i \le d_j\text{ and } 1\le k\le d_\ell,\\\notag
b^{j\ell}_{ik} &\coloneqq \begin{cases} f^{n_j}_{ip-k} &\qquad\qquad\quad\ \text{if $(jp-\ell)/m\in\Z_{\ge 0}$},\\0&\qquad\qquad\quad\ \text{otherwise}.\end{cases}
\end{align}
The diagonal blocks $B^{j,j}$ are square but the others typically will not be square, since the bound on $i$ depends on~$j$ while the bound on $k$ depends on~$\ell$.
We also note that there is at most one nonzero $B^{j\ell}$ in each row $j$, and in each column $\ell$ of $[B^{j\ell}]_{j\ell}$, since any nonzero $B^{j\ell}$ must have $\ell\equiv jp\bmod m$ (there will be no nonzero $B^{j\ell}$ for $j$ if no $\ell\le \mu$ satisfies $\ell\equiv jp\bmod m$; this happens, for example, when $j=1$ and $d=m=5$ with  $p\equiv 4\bmod 5$).
\begin{example}
For $m=5$ and $d=3$ we have $g=4$, and the $4\times 4$ matrix $A_p$ consists of $3\times 3 = 9$ blocks: one $2\times 2$, two $2\times 1$, two $1\times 2$, and four $1\times 1$.  For $k=\Fp$, the matrices $A_p$ for $p\equiv 1,2,3,4 \bmod 5$ are
\begin{footnotesize}
\[
\begin{pmatrix}
f_{p-1}^{(4p-4)/5} & f_{p-2}^{(4p-4)/5} & 0 & 0\\
f_{2p-1}^{(4p-4)/5} & f_{2p-2}^{(4p-4)/5} & 0 & 0\\
0 & 0 & f_{p-1}^{(3p-3)/5} & 0\\
0 & 0 & 0 & f_{p-1}^{(2p-2)/5}
\end{pmatrix},\  
\begin{pmatrix}
0 & 0 & f_{p-1}^{(4p-3)/5}  & 0\\
0 & 0 & f_{2p-1}^{(4p-3)/5} & 0\\
0 & 0 & 0 & 0\\
f_{p-1}^{(2p-4)/5} & f_{p-2}^{(2p-4)/5} & 0 & 0
\end{pmatrix},
\]
\[
\begin{pmatrix}
0 & 0 & 0 & f_{p-1}^{(4p-2)/5}\\
0 & 0 & 0 & f_{2p-1}^{(4p-2)/5}\\
f_{p-1}^{(3p-4)/5} & f_{p-2}^{(3p-4)/5} & 0 & 0\\
0 & 0 & 0 & 0
\end{pmatrix},\ 
\begin{pmatrix}
0 & 0 & 0 & 0\\
0 & 0 & 0 & 0\\
0 & 0 & 0 & f_{p-1}^{(3p-2)/5}\\
0 & 0 & f_{p-1}^{(2p-3)/5} & 0
\end{pmatrix}.
\]
\end{footnotesize}\hspace{-2pt}
For $m=3$ and $d=5$ we also have $g=4$ but now the $4\times 4$ matrix $A_p$ consists of $2\times 2 = 4$ blocks: one $3\times 3$, one $3\times 1$, one $1\times 3$, and one $1\times 1$. For $k=\Fp$ the matrices $A_p$ for $p\equiv1,2\bmod 3$ are
\begin{footnotesize}
\[
\begin{pmatrix}
f_{p-1}^{(2p-2)/3} & f_{p-2}^{(2p-2)/3} & f_{p-3}^{(2p-2)/3} & 0\\
f_{2p-1}^{(2p-2)/3} & f_{2p-2}^{(2p-2)/3} & f_{2p-3}^{(2p-2)/3} & 0\\
f_{3p-1}^{(2p-2)/3} & f_{3p-2}^{(2p-2)/3} & f_{3p-3}^{(2p-2)/3} & 0\\
0 & 0 & 0 & f_{p-1}^{(p-1)/3}
\end{pmatrix},\ 
\begin{pmatrix}
0 & 0 & 0 &  f_{p-1}^{(2p-1)/3}\\
0 & 0 & 0 & f_{2p-1}^{(2p-1)/3}\\
0 & 0 & 0 & f_{3p-1}^{(2p-1)/3}\\
f_{p-1}^{(p-2)/3} & f_{p-2}^{(p-2)/3} & f_{p-3}^{(p-2)/3} & 0
\end{pmatrix}.
\]
\end{footnotesize}\hspace{-2pt}
In both cases $\tr A_p = 0$ for $p\not\equiv 1\bmod m$, but this is not true in general (consider $m=4$ and $d=3$, for example).
\end{example}

The block form of the Cartier--Manin matrix $A_p$ given by \eqref{eq:B} implies the following theorem, which plays a key role in our algorithm for computing $A_p$ and may also be of independent interest.

\begin{theorem}\label{thm:supercartier}
Let $X\colon y^m=f(x)$ be a superelliptic curve over a perfect field of characteristic $p>0$ with $d\coloneqq \deg(f)$. Let~$\vec{\omega}$ be the basis of $\Omega_{k(X)}(0)$ given by Lemma~\ref{lem:basis}, and for $1\le j \le  m_1 = m-\lfloor m/d\rfloor-1$, let $\vec{\omega}_j\coloneqq \{\omega_{ij'}\in\boldsymbol\omega:j'=j\}$.  For $1\le j\le m_1$ the Cartier operator maps the subspace spanned by $\vec{\omega_j}$ to the subspace spanned by $\vec{\omega}_\ell$, with $\ell\equiv jp\bmod m$, and this action is given by the matrix $B^{j\ell}$ defined in \eqref{eq:B}.  In particular, when $p\equiv 1\bmod m$ the Cartier operator fixes each of the subspaces spanned by $\vec{\omega}_j$.
\end{theorem}
\begin{proof}
This is an immediate consequence of \eqref{eq:B}.
\end{proof}

\begin{remark}
In \cite[Lemma 5.1]{B01} Bouw gives formulas for the coefficients of the Hasse--Witt matrix of a general cyclic cover $Y\colon y^m=f(x)$ of $\mathbf{P}^1$ in terms of the (possibly repeated) roots of the polynomial $f\in k[x]$, where $k$ is an algebraically close field of characteristic $p$.  When $f$ is squarefree, Bouw's formulas agree with \eqref{eq:B}, after taking into account the transposition needed to get the Cartier--Manin matrix and a possible change of basis (I'm grateful to Wanlin Li and John Voight for bringing this to my attention).
One can compute analogs of the formulas in \eqref{eq:B} to handle $f$ that are not squarefree that take into account the multiplicities of its root, but we do not consider this case here.  Note that the genus of $Y$ and therefore the dimensions of $A_p$ will be less than that given by \eqref{eq:g} when~$f$ is not squarefree, so while the formulas may be more involved, the problem is computationally easier.
\end{remark}

\section{Linear recurrences}

The results of the previous section imply that to compute the Cartier--Manin matrix $A_p$ of a superelliptic curve $X\colon y^m=f(x)$ over $\Fp$ it suffices to compute certain coefficients of certain powers of $f(x)$.
In this section we derive linear recurrences that allow us to do this efficiently, both when $f\in \Fp[x]$ and when $f\in \Z[x]$ and we wish to compute certain coefficients of certain powers of the reduction of $f$ modulo many primes $p$.  In this section we generalize \cite[\S 2]{HS16}, which treated the case $m=2$, in which case $A_p=B$ consists of a single block $B^{11}$ (so $j=\ell=1$), the powers $f^n$ that appear in the matrix entries are always the same ($n=(p-1)/2$), and every prime $p\nmid m$ is congruent to $1$ modulo $m$.  Here we allow all of these parameters to vary.

Let $f\in \Z[x]$ be a squarefree polynomial of degree $d\ge 3$, which we shall write as $f(x)=x^ch(x)$ with $c=0,1$ and $h(0)\ne 0$ (note that $x^2\nmid f$).\footnote{The reader may wish to assume $c=0$ and $f=h$ on a first reading.}
Let $h(x)=\sum_{i=0}^r h_ix^i$, and for $n\ge 1$ let $h^n_i$ denote the coefficient of $x^i$ in $h(x)^n$.
As shown in \cite[\S 2]{HS16}, the identities $h^{n+1}=h\cdot h^n$ and $(h^{n+1})'=(n+1)h^n$ yield the linear relation
\begin{equation}\label{eq:fnk}
\sum_{i=0}^r((n+1)i-k)h_ih^n_{k-i} = 0,
\end{equation}
which is valid for all $k\in \Z$ and $n\in \Z_{\ge 0}$.
Observing that $n_j =((m-j)p-(m-\ell))/m$ is the exponent on $f$ in every entry of the nonzero block $B^{j\ell}$ defined in \eqref{eq:B}, let us set $n=n_j$ and rewrite \eqref{eq:fnk} as
\begin{equation}\label{eq:fnkp}
0 = \sum_{i=0}^r((m-j)p+\ell)i-mk)h_ih^{n_j}_{k-i}\equiv \sum_{i=0}^r(\ell i-mk)h_ih^{n_j}_{k-i}\bmod p,
\end{equation}
which is valid for all $k\in \Z$.  We now define
\[
v^{n_j}_k:=[h^{n_j}_{k-r+1},\ldots,h^{n_j}_k]\in\Z^r,
\]
and put $s\coloneqq p-1-cn_j$.  The entries of $v^n_s\bmod p$ suffice to compute the first row of block $B^{j\ell}$ in $A_p$; note that $n$ (and potentially $s$) depend on $j$ and will vary from block to block.  
We have $v^{n_j}_0=[0,\ldots,0,h_0^{n_j}]=h_0^{n_j}v_0^0$, where $v_0^0\coloneqq [0,\ldots,0,1]$.  Noting that $s<p$ and $p\nmid m$ and $p\nmid h_0$ (since $f$ is squarefree), solving for $h_k^n$ in \eqref{eq:fnkp} yields
\begin{equation}\label{eq:vn}
v^{n_j}_s \equiv \frac{v^{n_j}_0}{(mh_0)^s s!}\prod_{i=0}^{s-1} M^\ell_i \equiv m^{cn_j}h_0^{(c+1)n_j}(-1)^{cn_j+1}(cn_j)!v^0_0 \prod_{i=0}^{s-1} M^\ell_i\bmod p,
\end{equation}
where
\begin{equation}\label{eq:Mlk}
M^\ell_{i-1}:=\begin{bmatrix}0&\cdots & 0 &(\ell r-mi)h_r\\mih_0&\cdots&0&(\ell(r-1)-mi)h_{r-1}\\\vdots&\ddots&\vdots&\vdots\\0&\cdots&mih_0&(\ell-mi)h_1\end{bmatrix}
\end{equation}
is an integer matrix that depends on the integers $i,\ell,m$ and the polynomial $h$ of degree $r$, but is independent of~$p$.  This independence is the key to obtaining an average polynomial-time algorithm.

\begin{remark}\label{rem:lastrow}
Alternatively, if we define $w_k^n\coloneqq [h_{k+r-1}^{n_j},h_{k+r-2}^{n_j},\ldots, h_k^{n_j}]$ and $t\coloneqq d_jp-d_\ell-cn_j$, the entries of $w_t^n$ suffice to compute the last row of block $B^{j\ell}$ in $A_p$.  Equivalently, if we put $\tilde h(x)\coloneqq x^rh(1/x)$ (in other words, reverse the coefficients of $h$) and define $\tilde v_k^n$ in terms of $\tilde h^n$ as above, it suffices to compute $\tilde v_{\tilde s}^n$ where
\begin{align}\label{eq:vvn}
\tilde s \coloneqq rn_j - t &= dn_j-d_jp+d_\ell = p - 1 - \lfloor (dj\rem m)p/m\rfloor
\end{align}
When $m\nmid dj$ we will have $\tilde s < s$ if $c=0$ (and possibly even if $c=1$), in which case we can compute the last row more efficiently than the first.
\end{remark}

We have shown how to compute the first (or last) row of each of the blocks $B^{j\ell}$ that appear in the Cartier--Manin matrix of the superelliptic curve $X$ (either for $X/\Fp$ or for the reductions of $X/\Q$ modulo varying primes~$p$) by computing reductions of products of integer matrices modulo primes.  To compute the remaining rows in the same fashion would require working modulo powers of primes, which is something we wish to avoid.  In the next section we show how to efficiently reduce the computation of the remaining rows to the computation of the first row using translated curves, which allows us to always work modulo primes.

\section{Translation tricks}

Let $X\colon y^m=f(x)$ be a superelliptic curve over $\Fp$ of genus $g$, with $d\coloneqq \deg (f)$.  Let $A_p$ be the Cartier--Manin matrix $A_p$, and for $a\in \Fp$, let $A_p(a)$ be the Cartier--Manin matrix of the translated curve $X_a\colon y^m=f(x+a)$, whose blocks we denote $B^{j\ell}(a)$ with entries $b^{j\ell}_{ik}(a)$.  We omit the exponent $1/p$ that appears in \eqref{eq:B} because we are now working over $\Fp$.  The curve $X_a$ is isomorphic to $X$, which forces $A_p$ and $A_p(a)$ to be conjugate, but these matrices are typically not equal.  Our objective in this section is to show that we can compute $B^{j\ell}$ by solving a linear system that involves the entries that appear in just the first rows of $B^{j\ell}(a)$, where $a$ ranges over $d_j=d-\lfloor dj/m\rfloor-1$ 
distinct values of $a\in \Fp$.  Note that $B^{j\ell}$ has $ d_j$ rows and $ d_\ell$ columns, and we recall from \eqref{eq:B} that the $g\times g$ matrix $A_p$ is made up of $\mu^2$ blocks $B^{j\ell}$, where $\mu\coloneqq m_1= m-\lfloor m/d\rfloor-1$,
and we have $d_1+\cdots +d_\mu = g$.
We shall assume $p\ge d$, so that $ d_j< d$ distinct values of $a$ exist in $\Fp$; for $p<d$ we can easily compute $A_p$ directly from \eqref{eq:B}.

The results in this section generalize \cite[\S 5]{HS16}, which treated the case $m=2$, where $\mu=1$ and $A=B^{11}$. In our current setting $A_p$ consists of $\mu\times \mu$ rectangular blocks $B^{j\ell}$ that need not be square.

For $a\in \Fp$ and $1\le j\le \mu$ we define the upper triangular $ d_j\times  d_j$ matrix
\[
T^j(a)\coloneqq [t_{ik}^j(a)]_{ik},\qquad t^j_{ik}(a) \coloneqq \binom{k-1}{i-1}a^{k-i},\qquad 1\le i,k\le  d_j.
\]
We also define $T(a)$ to be the $g\times g$ block diagonal matrix with the matrices $T^j(a)$ on the diagonal, for $1\le j\le \mu$.
We note that $T^j(a)^{-1} = T^j(-a)$ and $T(a)^{-1}=T(-a)$, as the reader may verify (or see the proof below).

\begin{lemma}\label{lem:translation}
For $a\in \Fp$ we have $B^{j\ell}(a)T^\ell(a)=T^j(a)B^{j\ell}$ for all $1\le j,\ell \le \mu$, and $A_p(a)=T(a)A_pT(-a)$.
\end{lemma}
\begin{proof}
From the block structure of $A_p$ given by \eqref{eq:B} it is clear that the first statement implies the second.
Let $\vec{\omega}(a)=\{\omega_{ij}(a)\}$ be the basis given by Lemma~\ref{lem:basis} for $X_a$ and define $\vec{\omega}_j(a)\coloneqq \{\omega_{ij'}(a)\in\vec{\omega}:j'=j\}$.  
By Theorem~\ref{thm:supercartier}, the Cartier operator of $X$ maps the subspace spanned by $\vec{\omega}_j$ to the subspace spanned by $\vec{\omega}_\ell$ via the matrix $B^{j\ell}$, and the Cartier operator of $X_a$ maps the subspace spanned by $\vec{\omega}_j(a)$ to the subspace spanned by $\vec{\omega}_\ell(a)$ via the matrix $B^{j\ell}(a)$.  We just need to check that the matrices $T^\ell(a)$ and $T^j(a)$ correspond to the change of basis that occurs when we replace $x$ with $x+a$.
Noting that $d(x+a)=dx$ and $F(x+a)_y=F(x)_y$, we have
\begin{align*}
\omega_{kj}(a) = (x+a)^{k-1}y^{j-1}dx/F_y &= \sum_{i=1}^k\binom{k-1}{i-1}a^{k-i}x^{i-1}y^{j-1}dx/F_y\\
                                          &= \sum_{i=1}^k t^j_{ik}(a)\omega_{ij} = \sum_{i=1}^{ d_j}t^j_{ik}(a)\omega_{ij},
\end{align*}
and it follows that $\vec{\omega}_j(a)=T^j(a)\vec{\omega}_j$ (here we are viewing $\vec{\omega}_j$ and $\vec{\omega}_j(a)$ as column vectors).
This holds for any~$j$, including $\ell$, and the lemma follows.
\end{proof}

Let us now consider the computation of the $ d_j\times  d_\ell$ block $B^{j\ell}$.  Computing the $k$th entry in the first row of both sides of the identity $B^{j\ell}(a)T^\ell(a)=T^j(a)B^{j\ell}$ given by Lemma~\ref{lem:translation} yields
\[
\sum_{s=1}^{ d_\ell} b_{1s}^{j\ell}(a)t_{sk}^\ell(a) = \sum_{t=1}^{ d_j} t^j_{1t}(a)b^{j\ell}_{tk},
\]
which defines a linear equation with $ d_j$ unknowns $b^{j\ell}_{tk}$ in terms of the $b_{1s}^{j\ell}(a)$ and matrices $T^j(a)$ and $T^\ell(a)$ we assume are known.
Taking $ d_j$ distinct values of $a$, say $(a_1,\ldots,a_{ d_j})$ yields a linear system with $ d_j$ equations and $ d_j$ unknowns that we can solve because the $ d_j\times  d_j$ matrix $[t^j_{1t}(a_i)]_{it}=[a_i^{t-1}]_{it}$ is an invertible Vandermonde matrix $V(a_1,\ldots,a_{ d_j})$.  If we now define the $ d_j\times  d_\ell$ matrix
\begin{equation}\label{eq:B1}
B^{j\ell}_1(a_1,\ldots,a_{ d_j}) \coloneqq [b_{1s}^{j\ell}(a_i)]_{is}
\end{equation}
and let $W_1^{j\ell}$ be the $ d_j\times d_\ell$ matrix whose $i$th row is the $i$th row of $B_1^{j\ell}$ times $T^\ell(a_i)$, we can compute $B^{j\ell}$ as
\begin{equation}\label{eq:B1toB}
B^{j\ell} = V(a_1,\ldots,a_{ d_j})^{-1}W_1^{j\ell}.
\end{equation}

\begin{remark}
If we use Remark~\ref{rem:lastrow} to compute the last row of $B^{j\ell}$ we can compute the first row of $B^{j\ell}(a_i)$ for $a_1,\ldots,a_{d_j-1}$ and use \eqref{eq:B1toB} to deduce the last row of $W_1^{j\ell}$ from the last row of $B^{j\ell}$.  One might suppose that we could instead compute the last rows of the $B^{j\ell}(a_i)$ instead of their first rows, but this is not enough to deduce $B^{j\ell}$.
\end{remark}

\begin{lemma}\label{lem:Bbound}
Let $X\colon y^m=f(x)$ be a superelliptic curve over $\Fp$ with $d\coloneqq \deg(f)$, and let $a_1,\ldots,a_{d_1}$ be distinct elements of $\Fp$, where $ d_1 = d-\lfloor d/m\rfloor - 1$.  Given the matrices $B_1^{j\ell}(a_1,\ldots, a_{d_j})$ for $1\le j\le \mu=m_1=m-\lfloor m/d\rfloor-1$ with $\ell\equiv jp\bmod m$, we can compute the Cartier--Manin matrix $A_p$ of $X$ using $O(md^3)$ ring operations in $\Fp$ and space for $O(md+d^2)$ elements of $\Fp$.
\end{lemma}
\begin{proof}
We can compute $V(a_1,\ldots,a_{ d_j})^{-1}$ using $O( d_j^2)$ ring operations in $\Fp$ \cite{EF06}, and we can compute $T^\ell(a_i)$ in $O( d_j^2)$ ring operations (using $\binom{k}{i} = \binom{k-1}{i-1}+\binom{k-1}{i}$).  The computation of $W^{j\ell}$ requires $O( d_j d_\ell^2)$ $\Fp$-operations, and the matrix product in \eqref{eq:B1} uses $O( d_j^2 d_l)$ ring operations, so it takes $O( d_j^2 d_\ell+ d_\ell d_j^2) = O(d^3)$ ring operations to compute each $B^{j\ell}$.  There are at most $ \mu < m$ nonzero $B^{j\ell}$ to compute, so the total cost of computing $A_p$ given the matrices $B_1^{j\ell}(a_1,\ldots,a_{ d_j})$ is $O(md^3)$ ring operations in $\Fp$ while storing $O(md+d^2)$  elements of~$\Fp$.
\end{proof}

\begin{remark}
In terms of the genus $g\sim md/2$, the bound $O(md^3)$ is equivalent to $O(gd^2)$, which is always bounded by $O(g^3)$ but can be as small as $O(g)$ if $d=O(1)$ (this assumes we use a sparse representation of $A_p$).
\end{remark}

\begin{remark}
In addition to playing a key role in our strategy for computing $A_p$, using translated curves can improve performance, as noted in the case of hyperelliptic curves in \cite[\S 6.1]{HS16}.  In particular, if $f(x)$ has a rational root $a$ then the translated curve $X_a\colon y^m=f(x+a)=xh(x)$ will have $r=d-1$ and $c=d-r=1$, reducing both the dimension $r$ and number $t=p-1-cn$ of matrices $M^\ell_k$ that appear in the product in \eqref{eq:vn}.  It thus makes sense to choose our distinct translation points $a$ to be roots of $f(x)$ whenever possible.  Additionally, if $d$ is divisible by $m$ and $f(x)$ has a rational root $a$, we can replace $X$ with $X'\colon y^m=x^df(1/x+a)=g(x)$, where $g(x)$ has degree $d-1$, and this also applies to all translated curves $X'_{a'}$.  This applies both locally (over $\Fp$) and globally (over $\Q$).
\end{remark}

\section{Accumulating remainder trees and forests}
In this section we briefly recall some background on accumulating remainder trees and related complexity bounds.
Given a sequence of $r\times r$ matrices $M_0,\ldots, M_{N-1}$ and a sequence of coprime integers $m_0,\ldots,m_{N-1}$ we wish to compute the sequence of reduced partial products
\[
A_k\coloneqq M_0\cdots M_k\bmod m_k
\]
for $0\le k < N$. Let $M_{-1}\coloneqq M_N\coloneqq m_N\coloneqq 1$, and for $0\le k < N/2$ let $B_k\coloneqq M_{2k-1}M_{2k}$ and $b_k\coloneqq m_{2k}m_{2k+1}$.
If we recursively compute $C_k\coloneqq B_0\cdots B_k\bmod b_k = M_0\cdots M_{2k}\bmod m_{2k}m_{2k+1} $ for $0\le k < N/2$, we then have
\[
A_{2k} = C_{k}\bmod m_{2k}\qquad\text{and}\qquad A_{2k+1} =  C_kM_{2k+1} \bmod m_{2k+1},
\]
This is the \textsc{RemainderTree} algorithm given in \cite{HS14}. In our setting we actually want to compute products of the form $V \prod_k M_k$ that involve a row vector $V$, and for this problem the \textsc{RemainderForest} algorithm in \cite{HS14} achieves an improved time (and especially) space complexity by splitting the remainder tree into $2^\kappa$-subtrees, for a suitable choice of $\kappa$.  We record the following result from \cite{HS16}, in which $\|x\|$ denotes the logarithm of the largest absolute value appearing in nonzero integer matrix or integer vector $x$, including the case where $x$ is a single nonzero integer.

\begin{theorem}[\cite{HS16}]\label{thm:forest}
Given $V\in \Z^r$, $M_1,\ldots, M_N\in \Z^{r\times r}$, and $m_1,\ldots,m_N\in \Z$, let $n\coloneqq \lceil\log_2 N\rceil$, let $B$ be an upper bound on $\|\prod_{j=1}^N m_j\|$ such that $B/2^\kappa$ is an upper bound on $\|\prod_{j=st}^{st+t-1}m_j\|$ for $1\le s\le N/t$, where $t := 2^{n - \kappa}$.
Let $B'$ be an upper bound on $\|V\|$, and let~$H$ be an upper bound on $\|m_k\|,\|A_k\|$ for $1\le k\le N$, such that $\log r \le H$, and assume that $r = O(\log N)$.
The \textsc{RemainderForest} algorithm computes the vectors $V_k\coloneqq VM_1\cdots M_k\bmod m_k\in (\Z/m_k\Z)^r$ for $1\le k \le N$ in
\[
O(r^2\M(B+NH)(n-\kappa) + 2^\kappa r^2\M(B)+r\M(B'))
\]
time using space bounded by
\[
O(2^{-\kappa}r^2(B+NH)(n-\kappa)+r(B+B')).
\]
\end{theorem}
This theorem implies the following corollary, which is all we shall use.

\begin{corollary}\label{cor:forest}
Fix an absolute constant $c>0$.
Let $N$ be a positive integer, let $m_1,\ldots,m_N$ be a sequence of positive coprime integers with $\log m_k\le c\log N$, let $M_0,\ldots,M_{N-1}\in \Z^{r\times r}$ be integer matrices with $r,\|M_k\|\le c\log N$, and let $v_0\in \Z^r$ be a row vector with $\|v_0\|=cN\log N$.
We can compute the vectors
\[
v_k\coloneqq v_0\prod_{i=0}^{k-1} M_i\bmod m_k
\]
for $1\le k \le N$ in $O(r^2N\log^3\!N)$ time using $O(r^2N)$ space.
\end{corollary}
\begin{proof}
Applying Theorem~\ref{thm:forest} with $\kappa \coloneqq 2\log\log N$, $B=cN\log N$, $B'=c\log N$, and $H=c\log N$, yields an $O(r^2 \M(N\log N)\log N)$ time bound using $O(r^2N)$ space.
Now apply $\M(N)=O(N\log N)$ from \cite{HvdH19a}.
\end{proof}

\section{Algorithms}

We now give our algorithms for computing the Cartier--Manin matrix $A_p$ of a superelliptic curve $X/\Fp$ and for the reductions of a superelliptic curve $X/\Q$ modulo all good primes $p\le N$. In the descriptions below, expressions of the form ``$a \rem m$'' denote the least nonnegative remainder in Euclidean division of $a$ by $m$.
As above we assume $X$ is defined by $y^m=f(x)$ with $f(x)$ squarefree of degree $d\ge 3$.  We define $ \mu\coloneqq m-\lfloor m/d\rfloor -1$, and for $1\le j\le \mu$ we put $ d_j\coloneqq d - \lfloor dj/m\rfloor -1$, with $ d_1\ge  d_2\ge \cdots  d_{\mu}$ as in \eqref{eq:dj}.  Recall that the genus $g$ of $X$ is $g\coloneqq ((d-2)(m-1)+m-\gcd(m,d))/2$, as in \eqref{eq:g}.

\bigskip

\noindent
\textbf{Algorithm} \textsc{ComputeCartierManinMatrix}
\vspace{2pt}

\noindent
Given $m\ge 2$ and squarefree $f\in \Fp[x]$ of degree $3\le d \le p$ with $p\nmid m$, compute the Cartier--Manin matrix $A_p\in \Fp^{g\times g}$ of $X\colon y^m=f(x)$ as follows:
\smallskip

\begin{enumerate}[1.]
\setlength{\itemsep}{3pt}
\item Fix distinct $a_1,\ldots,a_{ d_1}\in \Fp$ that include as many roots of $f(x)$ as possible.
\item For $j$ from $1$ to $ \mu$ such that $\ell\coloneqq jp \rem m \le \mu$:
\vspace{3pt}
\begin{enumerate}[a.]
\setlength{\itemsep}{3pt}
\item For $i$ from $1$ to $ d_j$:
\vspace{3pt}
\begin{enumerate}[i.]
\setlength{\itemsep}{3pt}
\item Let $f(x+a_i)=x^ch(x)\in \Fp[x]$ with $c\in \{0,1\}$ and put $r\coloneqq \deg(h)$.
\item Set $n\coloneqq ((m-j)p-(m-\ell))/m\in \Z$ and $s\coloneqq p-1-cn$.
\item Compute $w_s\coloneqq v_0^0\prod_{i=0}^{s-1} M_i^\ell\in \Fp^r$, with $M_i^{\ell}\in \Fp^{r\times r}$ as in \eqref{eq:Mlk},
and $u_s \coloneqq s!\in\Fp$.
\item Compute $\alpha\coloneqq v_s^n = m^{-s}h_0^{n-s}u_s^{-1} w_s \in \F_p^r$ via \eqref{eq:vn}.
\item Let $b_1^{j\ell}(a_i)\coloneqq [\alpha_r,\alpha_{r-1},\ldots \alpha_{r- d_\ell+1}]\in \F_p^{ d_\ell}$.
\end{enumerate}
\item Let $B_1^{j\ell}\in \Fp^{ d_j\times d_\ell}$ be the matrix with $i$th row $b_1^{j\ell}(a_i)$ as in \eqref{eq:B1} and use $B_1^{j\ell}$ to compute $B^{j\ell} \in \Fp^{ d_j\times d_\ell}$ via \eqref{eq:B1toB}.
\end{enumerate}
\item Output $A_p\coloneqq [B^{j\ell}]_{j\ell}\in \Fp^{g\times g}$ defined as in \eqref{eq:B}, with $B^{j\ell}\coloneqq 0$ for $\ell\not\equiv jp\bmod m$.
\end{enumerate}
\bigskip

There are two ways to compute $w_s$ in step iii.  One is to compute $s$ vector-matrix products $w_{i+1}\coloneqq w_iM_i^\ell$ starting with $w_0\coloneqq [0,\ldots,0,1]\in \Fp^r$, which can be accomplished with $O(pr)$ ring operations in $\Fp$, each of which takes time $\M(\log p)=O(\log p\log\log p)$ via \cite{HvdH19b}, yielding a time complexity of $O(rp\log p\log\log p)$ and a space complexity of $O(r\log p)$  (note that $M_i^{\ell}$ has only $2r-1$ nonzero entries).  Alternatively one can use the Bostan-Gaudry-Schost algorithm~\cite{BGS07}, which uses an optimized interpolation/evaluation approach to compute products of matrices over polynomial rings evaluated along an arithmetic progression; in our setting we view the $M_i^{\ell}$ as matrices of linear polynomials in~$i$ evaluated along the arithmetic progression $i=0,1,2,\ldots, s-1$.  This has a bit complexity of $O(r^2p^{1/2}\log p (r^{\omega-2}\log\log p + \log p)$ using $O(r^2p^{1/2}\log p)$ space via \cite[Thm.\,8]{BGS07} and \cite{HvdH19b}, and we can similarly compute $u_s=s!$ (but note that $u_s=-1$ in the typical case where $c= 0$).  Note that $r$ is either $d$ or $d-1$, so we can replace $r$ with $d$ in these bounds. This analysis leads to Theorem~\ref{thm:modp} given in the introduction, which we restate here for convenience.

\firstmodp*
\begin{proof}
Excluding step iii, applying Lemma~\ref{lem:Bbound} with an $O(\log p\log\log p)$ cost per ring operation in $\Fp$ yields a time complexity of $O(md^3\log p\log\log)$ using $O((md+d^2)\log p)$ space.  Step iii is executed $O(md)$ times and we can bound the time and space by simply multiplying the estimates above by $md$.  Taking the maximum of the complexity of Lemma 3 and the time spent in step iii yields the theorem, after noting that $O(md^2p^{1/2}\log p)$ dominates all of the space bounds.
\end{proof}

We now present our main result, an average polynomial-time algorithm to compute the Cartier--Manin matrices of the reductions of a superelliptic curve $X/\Q$ at all good primes $p\le N$.
\smallskip

\noindent
\textbf{Algorithm} \textsc{ComputeCartierManinMatrices}
\vspace{2pt}

\noindent
Given $m\ge 2$ and squarefree $f\in \Z[x]$ of degree $d\ge 3$, compute the Cartier--Manin matrices $A_p$ of the reductions of $X\colon y^m=f(x)$ modulo primes $p\le N$ with $p\nmid m\lc(f)\disc(f)$ as follows:
\smallskip

\begin{enumerate}[1.]
\setlength{\itemsep}{3pt}
\item For primes $p\le N$ with $p\nmid m\lc(f)\disc(f)$ initialize $A_p\in \Fp^{g\times g}$ to the zero matrix.
\item Fix distinct $a_1,\ldots,a_{ d_1}\in \Z$ that include as many roots of $f$ as possible.
\item For each pair of integers $j,\ell\in [1,\mu]$:
\vspace{3pt}
\begin{enumerate}[a.]
\setlength{\itemsep}{3pt}
\item Compute the set $P = \{p_1,p_2,\cdots\}$ of primes $p\le N$ with $jp\equiv \ell\bmod m$\\
such that $p\nmid m\lc(f)\disc(f)$ and $a_1,\ldots,a_{ d_1}$ are distinct modulo $p$.
\item If the set $P$ is empty proceed to the next pair $j,\ell$.
\item For $i$ from $1$ to $ d_j$:
\vspace{3pt}
\begin{enumerate}[i.]
\setlength{\itemsep}{3pt}
\item Let $f(x+a_i)=x^ch(x)\in \Z[x]$ with $c\in \{0,1\}$ and put $r\coloneqq \deg(h)$.
\item Let $N'\coloneqq N$ if $c=0$ and $N'\coloneqq \lfloor(jN-\ell)/m)\rfloor$ otherwise.
\item Define coprime moduli $m_1,\ldots,m_{N'}$ as follows:\\
\phantom{D} If $c=0$ then $m_k\coloneqq k+1$ for $k+1\in P$.\\
\phantom{D} If $c=1$ then $m_k\coloneqq (mk+\ell)/j$ for $(mk+\ell)/j\in P$.\\
\phantom{D} For any $m_k$ not defined above, let $m_k\coloneqq 1$.\\
For $p\in P$ let $k(p)$ denote the index $k$ of the $m_k$ for which $m_k=p$.
\item Compute $w_k\coloneqq v_0^0\prod_{i=0}^{k-1} M_i^\ell \bmod m_k$ and $u_k\coloneqq k!\bmod m_k$ for $1\le k \le N'$ as in Corollary~\ref{cor:forest}.
\item For $p\in P$ use $w_{k(p)}$, $u_{k(p)}$ to compute $b_1^{j\ell}(a_i)\in \Fp^{ d_\ell}$ as in \textsc{ComputeCartierManinMatrix}.
\end{enumerate}
\item For $p\in P$, let $B_1^{j\ell}\in \Fp^{ d_j\times d_\ell}$ have rows $b_1^{j\ell}(a_i)\in \Fp^{ d_\ell}$ as in \eqref{eq:B1}, use $B_1^{j\ell}$ to compute $B^{j\ell} \in \Fp^{ d_j\times d_\ell}$ via~\eqref{eq:B1toB}, and set the $j,\ell$ block of $A_p$ to $B^{j\ell}$ as in \eqref{eq:B}.
\end{enumerate}
\item Let $S$ be the set of primes $p\le N$ satisfying $p\nmid m\lc(f)\disc(f)$ for which the $a_1,\ldots a_{ d_1}$ are not distinct modulo~$p$.
For $p\in S$ compute $A_p$ using algorithm \textsc{ComputeCartierManinMatrix} if $p\ge d$ and otherwise compute $A_p$ directly from \eqref{eq:B} by extracting coefficients of powers of $f\in \Fp[x]$.
\item Output $A_p\in \Fp^{g\times g}$ for all primes $p \le N$ such that $p\nmid m \lc(f)\disc(f)$.
\end{enumerate}
\smallskip

\begin{remark}\label{rem:traceonly}
To compute Frobenius traces $a_p\in \Z$, we modify step 3 to loop over integers $j=\ell\in [1,\mu]$ and output just the traces of the $A_p$ in step 5.
This gives the traces of Frobenius $a_p\bmod p$.  For $p>16g^2$ these determine $a_p\in \Z$, by the Weil bounds, and for $p\le 16g^2$ we can compute $a_p=p+1-\#X(\Fp)$ by enumerating values of $f(x)$ and looking them up in a precomputed table of $m$th powers.
\end{remark}

\begin{remark}
The inner loop in step 3.c is executed (up to) $\mu g$ times.  Each of these computations is completely independent of the others, which makes it easy to efficiently distribute the work across $\mu g$ threads.
In principal one can also parallelize the integer matrix multiplications performed by the \textsc{RemainderForest} algorithm in step~iv, but in practice it is extremely difficult to do this efficiently.
\end{remark}

We now prove Theorem~\ref{thm:avgpoly}, which we restate for convenience.

\firstavgpoly*
\begin{proof}
The total time to compute all the sets $P$ using a sieve is bounded by $O(N\log N)$ time using $O(N)$ space, and this also bounds the total time and space for steps i, ii, iii, under our assumption that $m,d,\|f\|=O(\log N)$.  Corollary~\ref{cor:forest} yields an $O(d^2N\log^3N)$ bound on each of the $O(m^2d)$ iterations of step iv.
This yields the claimed time bound of $O(m^2d^3N\log^3N)$ for step 3.c, which we claim dominates.  Lemma~\ref{lem:Bbound} implies that the total cost of step 3.d is bounded by $O(\pi(N)m^2d^3\log N)$, which is negligible, as is the cost of the rest of the algorithm.  Note that the cardinality of the set $S$ in step 4 is at worst quadratic in $d$ and $\log(N)$ under our assumption $\|f\|=O(\log N)$, so we can easily afford the calls to \textsc{ComputeCartierManinMatrix} and use a brute force approach to compute~$A_p$ for primes $p < d$ of good reduction.

The space bound follows from the bound in Corollary~\ref{cor:forest}, which covers steps iv (it is easy to see that all of the other steps fit within the claimed bound).

To compute Frobenius traces $a_p\in\Z$ we apply Remark~\ref{rem:traceonly} and note that restricting to $j=\ell$ in step 3 reduces the number of iterations of the main loop by a factor of $m$.  The cost of computing $\#X(\Fp)$ by looking up values of $f(x)$ in a table of $m$th powers is $O(pd)$ ring operations in $\Fp$.   The total time to compute $a_p=p+1-\#X(\Fp)$ for good $p\le 16g^2$ is then $O(dg^2\pi(g^2)\log g\log\log g)=O(d(\log N)^4\log\log N)$, which is negligible.
\end{proof}

\section{Performance comparison}\label{sec:perf}

Tables \ref{tab:perfhard} and \ref{tab:perfeasy} compare the performance of the average polynomial-time algorithm \textsc{ComputeCartierManinMatrices} with the $\tilde{O}(p^{1/2})$ algorithm for computing zeta functions of cyclic covers implemented in Sage version~9.0.  The Sage implementation provides the function \texttt{CyclicCover} which takes an integer $m$ and a squarefree polynomial $f\in \Fp[x]$ and returns an object that represents a superelliptic curve $y^m=f(x)$ over $\Fp$.
Invoking the \texttt{frobenius\_matrix} method of this object with the $p$-adic precision set to $1$ yields a matrix that encodes essentially the same information as the Cartier--Manin matrix $A_p$; in particular it determines the $p$-rank of $X$ and its zeta function modulo $p$.

Each table lists the genus $g$ and invariants $m$ and $d$ of a superelliptic curve $X\colon y^m=f(x)$ defined over $\Q$ with $f\in \Z[x]$ of degree $d$. There is a row for every pair $m\ge 2$ and $d\ge 3$ for which $m^2d^3\le 6^5$, which includes all superelliptic curves of genus $g\le 5$ as well as plane quintics and sextics, and other curves of genus up to 15.  The times listed are average times in milliseconds for primes $p\le N$ for increasing values of $N$.  For each $N$ three times are listed: one to compute Frobenius matrices using Sage, one to compute Cartier--Manin matrices using algorithm \textsc{ComputeCartierManinMatrices}, and one to to compute Frobenius traces via Remark~\ref{rem:traceonly}.  For the Sage timings we only computed Frobenius matrices for every $n$th good prime $p\le N$ with $n$ chosen so that the computation would complete in less than a day (many of the computations would have taken months otherwise).

In Table~\ref{tab:perfhard} we show timings with $f\in \Z[x]$ having coefficients $f_{d+1-n}\coloneqq p_n$ for $1\le n\le d$, where $p_n$ is the $n$th prime.  These polynomials are all irreducible, so our algorithm was unable to choose any $a_i$ to be roots of $f$; this is the generic situation, and the worst case for our algorithm.
In Table~\ref{tab:perfeasy} we show timings with $f\in \Z[x]$ a product of linear factors, which represents the best case for our algorithm.

\begin{table}[!htb]
\small
\begin{tabular}{@{}rrrrrrrrrrrrrrrrrrrrr@{}}
&&&&\multicolumn{3}{c}{$N=2^{16}$}&&\multicolumn{3}{c}{$N=2^{20}$}&&\multicolumn{3}{c}{$N=2^{24}$}&&\multicolumn{3}{c}{$N=2^{28}$}\\
\cmidrule(r){5-7}\cmidrule(r){9-11}\cmidrule(r){13-15}\cmidrule(r){17-19}
$g$ & $m$ & $d$ && sage & matrix & trace && sage & matrix & trace && sage & matrix & trace && sage & matrix & trace\\
\midrule
1 & 2 & 3 && 21 & 0.01 & 0.01 && 27 & 0.05 & 0.05 && 67 & 0.13 & 0.13 &&  230 & 0.30 & 0.30 \\
1 & 2 & 4 && 27 & 0.04 & 0.04 && 41 & 0.17 & 0.16 && 120 & 0.42 & 0.42 &&  454 & 0.95 & 0.93 \\
1 & 3 & 3 && 27 & 0.02 & 0.02 && 46 & 0.08 & 0.08 && 141 & 0.20 & 0.20 &&  499 & 0.48 & 0.49 \\
2 & 2 & 5 && 30 & 0.08 & 0.08 && 55 & 0.38 & 0.38 && 163 & 0.92 & 0.92 &&  580 & 2.02 & 2.01 \\
2 & 2 & 6 && 42 & 0.16 & 0.16 && 83 & 0.73 & 0.74 && 280 & 1.77 & 1.77 &&  1070 & 3.89 & 3.92 \\
3 & 2 & 7 && 53 & 0.24 & 0.24 && 112 & 1.30 & 1.29 && 307 & 3.19 & 3.12 &&  1217 & 6.47 & 6.71 \\
3 & 2 & 8 && 74 & 0.34 & 0.34 && 169 & 2.15 & 2.07 && 528 & 5.02 & 4.94 &&  2106 & 10.20 & 10.57 \\
3 & 3 & 4 && 34 & 0.10 & 0.05 && 61 & 0.53 & 0.26 && 178 & 1.38 & 0.70 &&  702 & 3.14 & 1.63 \\
3 & 4 & 3 && 32 & 0.03 & 0.03 && 58 & 0.14 & 0.15 && 165 & 0.37 & 0.37 &&  601 & 0.89 & 0.89 \\
3 & 4 & 4 && 49 & 0.09 & 0.09 && 101 & 0.44 & 0.44 && 343 & 1.14 & 1.14 &&  1283 & 2.55 & 2.63 \\
4 & 2 & 9 && 96 & 0.43 & 0.44 && 194 & 3.22 & 3.24 && 576 & 7.65 & 7.70 &&  2214 & 16.12 & 15.90 \\
4 & 2 & 10 && 138 & 0.55 & 0.55 && 319 & 4.78 & 4.65 && 974 & 11.10 & 10.98 &&  3693 & 22.13 & 22.79 \\
4 & 3 & 5 && 47 & 0.22 & 0.11 && 93 & 1.29 & 0.65 && 287 & 3.37 & 1.67 &&  1105 & 7.64 & 3.68 \\
4 & 3 & 6 && 71 & 0.36 & 0.18 && 152 & 2.59 & 1.28 && 535 & 6.34 & 3.20 &&  2121 & 14.04 & 7.07 \\
4 & 5 & 3 && 37 & 0.08 & 0.03 && 68 & 0.40 & 0.13 && 200 & 1.19 & 0.40 &&  778 & 2.96 & 0.99 \\
4 & 6 & 3 && 49 & 0.05 & 0.06 && 112 & 0.24 & 0.24 && 313 & 0.64 & 0.64 &&  1184 & 1.53 & 1.53 \\
5 & 2 & 11 && 170 & 0.71 & 0.70 && 361 & 7.04 & 7.06 && 1024 & 16.57 & 16.30 &&  3695 & 33.61 & 33.32 \\
5 & 2 & 12 && 263 & 0.85 & 0.86 && 555 & 9.56 & 9.54 && 1537 & 21.84 & 22.23 &&  5820 & 45.98 & 45.65 \\
6 & 3 & 7 && 90 & 0.53 & 0.27 && 200 & 4.61 & 2.32 && 632 & 11.53 & 5.52 &&  2360 & 24.18 & 12.18 \\
6 & 4 & 5 && 63 & 0.31 & 0.20 && 130 & 1.71 & 1.08 && 424 & 4.37 & 2.73 &&  1658 & 9.86 & 5.88 \\
6 & 5 & 4 && 55 & 0.21 & 0.07 && 113 & 1.29 & 0.42 && 344 & 3.76 & 1.25 &&  1358 & 9.08 & 3.03 \\
6 & 5 & 5 && 90 & 0.39 & 0.13 && 201 & 3.06 & 1.02 && 671 & 8.98 & 2.92 &&  2749 & 19.39 & 6.64 \\
6 & 7 & 3 && 49 & 0.14 & 0.04 && 94 & 0.68 & 0.17 && 290 & 2.24 & 0.56 &&  1146 & 5.57 & 1.39 \\
7 & 3 & 8 && 134 & 0.75 & 0.38 && 294 & 8.17 & 4.05 && 835 & 19.07 & 9.38 &&  3279 & 40.32 & 20.49 \\
7 & 3 & 9 && 187 & 0.99 & 0.50 && 437 & 12.77 & 6.32 && 1462 & 28.54 & 14.50 &&  5567 & 61.82 & 29.67 \\
7 & 4 & 6 && 102 & 0.52 & 0.34 && 232 & 3.42 & 2.12 && 806 & 8.58 & 5.21 &&  3160 & 18.99 & 11.54 \\
7 & 6 & 4 && 75 & 0.21 & 0.15 && 153 & 1.08 & 0.77 && 524 & 2.79 & 2.00 &&  2112 & 6.46 & 4.55 \\
7 & 8 & 3 && 55 & 0.13 & 0.06 && 111 & 0.60 & 0.29 && 366 & 1.72 & 0.83 &&  1333 & 4.32 & 2.00 \\
7 & 9 & 3 && 67 & 0.16 & 0.06 && 140 & 0.82 & 0.26 && 479 & 2.64 & 0.82 &&  1870 & 6.77 & 2.03 \\
9 & 4 & 7 && 139 & 0.80 & 0.53 && 302 & 6.49 & 3.94 && 941 & 15.10 & 9.42 &&  3566 & 32.97 & 20.43 \\
9 & 7 & 4 && 75 & 0.40 & 0.08 && 156 & 2.77 & 0.56 && 510 & 9.14 & 1.78 &&  2012 & 20.90 & 4.21 \\
9 & 8 & 4 && 92 & 0.32 & 0.17 && 231 & 1.85 & 0.92 && 720 & 5.43 & 2.57 &&  2941 & 12.58 & 6.12 \\
9 & 10 & 3 && 65 & 0.16 & 0.08 && 137 & 0.76 & 0.34 && 429 & 2.29 & 1.01 &&  1694 & 5.82 & 2.50 \\
10 & 5 & 6 && 114 & 0.80 & 0.20 && 265 & 8.08 & 2.02 && 840 & 22.89 & 5.62 &&  3256 & 51.62 & 12.42 \\
10 & 6 & 5 && 97 & 0.43 & 0.32 && 206 & 2.51 & 1.83 && 701 & 6.28 & 4.61 &&  2700 & 14.07 & 9.88 \\
10 & 6 & 6 && 175 & 0.71 & 0.53 && 379 & 5.05 & 3.49 && 1278 & 11.95 & 8.59 &&  5202 & 26.43 & 18.72 \\
10 & 11 & 3 && 73 & 0.30 & 0.05 && 158 & 1.77 & 0.25 && 501 & 6.11 & 0.88 &&  1878 & 15.32 & 2.12 \\
10 & 12 & 3 && 91 & 0.17 & 0.11 && 187 & 0.80 & 0.49 && 636 & 2.35 & 1.39 &&  2558 & 5.87 & 3.45 \\
12 & 7 & 5 && 118 & 0.73 & 0.15 && 246 & 6.75 & 1.33 && 840 & 20.80 & 4.13 &&  3228 & 48.09 & 9.23 \\
12 & 9 & 4 && 94 & 0.43 & 0.14 && 199 & 2.88 & 0.87 && 657 & 8.87 & 2.64 &&  2655 & 21.75 & 6.24 \\
12 & 13 & 3 && 94 & 0.38 & 0.05 && 175 & 2.43 & 0.29 && 616 & 8.24 & 1.03 &&  2244 & 20.02 & 2.49 \\
13 & 10 & 4 && 117 & 0.47 & 0.19 && 264 & 2.90 & 1.09 && 1008 & 8.62 & 3.17 &&  3762 & 20.08 & 7.47 \\
13 & 14 & 3 && 90 & 0.30 & 0.09 && 193 & 1.58 & 0.43 && 619 & 5.01 & 1.36 &&  2430 & 12.79 & 3.40 \\
13 & 15 & 3 && 109 & 0.31 & 0.09 && 235 & 1.69 & 0.46 && 811 & 5.54 & 1.45 &&  3238 & 13.99 & 3.72 \\
15 & 11 & 4 && 111 & 0.81 & 0.10 && 252 & 6.29 & 0.79 && 839 & 22.76 & 2.84 &&  3334 & 52.85 & 6.59 \\
15 & 16 & 3 && 110 & 0.32 & 0.11 && 223 & 1.79 & 0.53 && 733 & 5.66 & 1.63 &&  2805 & 14.16 & 4.13 \\
\bottomrule
\end{tabular}
\bigskip
\bigskip

\caption{\small Comparison with $\tilde O(p^{1/2})$ Sage 9.0 implementation \cite{ABCMT19} for superelliptic curves $y^m=f(x)$ where $f\in \Z[x]$ is irreducible of degree $d$. Times are millisecond averages per prime $p\le N$ for a single thread running on a 2.8GHz Cascade Lake Intel CPU. The \texttt{sage} column lists the average time to execute \texttt{CyclicCover(m,f.change\_ring(GF(p)).frobenius\_matrix(1)} in Sage 9.0, the matrix column lists the average time to compute the Cartier--Manin matrix modulo $p$ using algorithm \textsc{ComputeCartierManinMatrices}, and the trace column is the average time to compute the trace of Frobenius via Remark~\ref{rem:traceonly}.}\label{tab:perfhard}
\end{table}

\begin{table}[!htb]
\small
\begin{tabular}{@{}rrrrrrrrrrrrrrrrrrrrr@{}}
&&&&\multicolumn{3}{c}{$N=2^{16}$}&&\multicolumn{3}{c}{$N=2^{20}$}&&\multicolumn{3}{c}{$N=2^{24}$}&&\multicolumn{3}{c}{$N=2^{28}$}\\
\cmidrule(r){5-7}\cmidrule(r){9-11}\cmidrule(r){13-15}\cmidrule(r){17-19}
$g$ & $m$ & $d$ && sage & matrix & trace && sage & matrix & trace && sage & matrix & trace && sage & matrix & trace\\
\midrule
1 & 2 & 3 && 20 & 0.01 & 0.01 && 28 & 0.01 & 0.01 && 73 & 0.04 & 0.04 &&  230 & 0.09 & 0.08 \\
1 & 2 & 4 && 26 & 0.01 & 0.01 && 43 & 0.04 & 0.05 && 119 & 0.12 & 0.12 &&  456 & 0.28 & 0.27 \\
1 & 3 & 3 && 27 & 0.00 & 0.00 && 45 & 0.01 & 0.01 && 131 & 0.02 & 0.02 &&  500 & 0.05 & 0.05 \\
2 & 2 & 5 && 29 & 0.03 & 0.03 && 53 & 0.11 & 0.12 && 151 & 0.31 & 0.30 &&  583 & 0.72 & 0.72 \\
2 & 2 & 6 && 41 & 0.05 & 0.06 && 84 & 0.26 & 0.28 && 267 & 0.66 & 0.64 &&  1071 & 1.40 & 1.40 \\
3 & 2 & 7 && 53 & 0.10 & 0.10 && 116 & 0.55 & 0.54 && 311 & 1.22 & 1.20 &&  1219 & 2.58 & 2.59 \\
3 & 2 & 8 && 77 & 0.13 & 0.14 && 164 & 0.94 & 0.92 && 532 & 2.06 & 2.04 &&  2094 & 4.19 & 4.23 \\
3 & 3 & 4 && 34 & 0.03 & 0.02 && 62 & 0.14 & 0.07 && 184 & 0.41 & 0.20 &&  701 & 0.96 & 0.47 \\
3 & 4 & 3 && 31 & 0.01 & 0.01 && 55 & 0.03 & 0.03 && 157 & 0.08 & 0.08 &&  605 & 0.20 & 0.20 \\
3 & 4 & 4 && 48 & 0.02 & 0.02 && 103 & 0.08 & 0.09 && 334 & 0.23 & 0.23 &&  1286 & 0.55 & 0.54 \\
4 & 2 & 9 && 94 & 0.19 & 0.19 && 199 & 1.50 & 1.47 && 586 & 3.48 & 3.41 &&  2232 & 7.10 & 7.12 \\
4 & 2 & 10 && 135 & 0.25 & 0.25 && 295 & 2.30 & 2.29 && 942 & 5.37 & 5.24 &&  3816 & 10.53 & 10.37 \\
4 & 3 & 5 && 46 & 0.07 & 0.04 && 92 & 0.38 & 0.19 && 283 & 1.06 & 0.51 &&  1111 & 2.40 & 1.21 \\
4 & 3 & 6 && 72 & 0.12 & 0.06 && 153 & 0.79 & 0.41 && 529 & 1.85 & 0.91 &&  2098 & 3.96 & 1.99 \\
4 & 5 & 3 && 38 & 0.02 & 0.01 && 68 & 0.05 & 0.02 && 202 & 0.16 & 0.05 &&  780 & 0.39 & 0.13 \\
4 & 6 & 3 && 48 & 0.01 & 0.01 && 95 & 0.03 & 0.03 && 301 & 0.09 & 0.09 &&  1186 & 0.22 & 0.21 \\
5 & 2 & 11 && 171 & 0.31 & 0.31 && 354 & 3.45 & 3.46 && 977 & 7.85 & 7.87 &&  3682 & 15.94 & 15.85 \\
5 & 2 & 12 && 246 & 0.37 & 0.40 && 530 & 5.11 & 5.12 && 1543 & 11.30 & 11.17 &&  5857 & 22.61 & 22.62 \\
6 & 3 & 7 && 89 & 0.19 & 0.10 && 192 & 1.47 & 0.72 && 605 & 3.57 & 1.78 &&  2361 & 7.67 & 3.79 \\
6 & 4 & 5 && 64 & 0.08 & 0.05 && 136 & 0.32 & 0.25 && 416 & 0.94 & 0.61 &&  1660 & 2.17 & 1.43 \\
6 & 5 & 4 && 55 & 0.07 & 0.03 && 108 & 0.30 & 0.10 && 348 & 1.00 & 0.32 &&  1369 & 2.43 & 0.81 \\
6 & 5 & 5 && 92 & 0.09 & 0.03 && 196 & 0.52 & 0.15 && 710 & 1.48 & 0.48 &&  2755 & 3.49 & 1.16 \\
6 & 7 & 3 && 50 & 0.03 & 0.01 && 96 & 0.06 & 0.02 && 296 & 0.23 & 0.06 &&  1146 & 0.63 & 0.15 \\
7 & 3 & 8 && 125 & 0.28 & 0.16 && 276 & 3.05 & 1.54 && 836 & 7.04 & 3.49 &&  3234 & 15.09 & 7.64 \\
7 & 3 & 9 && 193 & 0.35 & 0.18 && 427 & 4.09 & 2.16 && 1409 & 9.28 & 4.74 &&  5551 & 21.20 & 10.35 \\
7 & 4 & 6 && 98 & 0.17 & 0.12 && 227 & 0.98 & 0.65 && 774 & 2.30 & 1.48 &&  3143 & 5.26 & 3.33 \\
7 & 6 & 4 && 70 & 0.06 & 0.04 && 155 & 0.23 & 0.17 && 525 & 0.66 & 0.44 &&  2108 & 1.53 & 1.04 \\
7 & 8 & 3 && 55 & 0.02 & 0.02 && 111 & 0.06 & 0.04 && 343 & 0.20 & 0.12 &&  1333 & 0.51 & 0.30 \\
7 & 9 & 3 && 69 & 0.03 & 0.01 && 141 & 0.08 & 0.03 && 476 & 0.28 & 0.09 &&  1876 & 0.76 & 0.23 \\
9 & 4 & 7 && 127 & 0.30 & 0.19 && 289 & 1.85 & 1.23 && 917 & 4.56 & 2.88 &&  3555 & 10.28 & 6.23 \\
9 & 7 & 4 && 71 & 0.12 & 0.03 && 156 & 0.61 & 0.10 && 509 & 1.78 & 0.35 &&  2007 & 4.47 & 0.88 \\
9 & 8 & 4 && 93 & 0.09 & 0.04 && 211 & 0.33 & 0.18 && 752 & 1.05 & 0.50 &&  2946 & 2.64 & 1.23 \\
9 & 10 & 3 && 76 & 0.04 & 0.02 && 139 & 0.08 & 0.04 && 430 & 0.26 & 0.12 &&  1694 & 0.66 & 0.31 \\
10 & 5 & 6 && 115 & 0.25 & 0.07 && 253 & 2.08 & 0.52 && 825 & 5.96 & 1.49 &&  3265 & 13.97 & 3.37 \\
10 & 6 & 5 && 101 & 0.13 & 0.09 && 213 & 0.68 & 0.42 && 676 & 1.61 & 1.06 &&  2693 & 3.83 & 2.43 \\
10 & 6 & 6 && 155 & 0.19 & 0.15 && 365 & 1.23 & 0.86 && 1276 & 2.94 & 2.00 &&  5195 & 6.46 & 4.34 \\
10 & 11 & 3 && 74 & 0.05 & 0.01 && 154 & 0.14 & 0.02 && 477 & 0.52 & 0.08 &&  1878 & 1.48 & 0.21 \\
10 & 12 & 3 && 87 & 0.03 & 0.02 && 189 & 0.08 & 0.07 && 640 & 0.26 & 0.17 &&  2552 & 0.63 & 0.42 \\
12 & 7 & 5 && 113 & 0.18 & 0.04 && 242 & 1.22 & 0.24 && 879 & 3.99 & 0.77 &&  3227 & 9.89 & 1.93 \\
12 & 9 & 4 && 95 & 0.11 & 0.04 && 204 & 0.60 & 0.17 && 672 & 1.66 & 0.52 &&  2663 & 4.30 & 1.26 \\
12 & 13 & 3 && 83 & 0.06 & 0.01 && 175 & 0.19 & 0.02 && 569 & 0.71 & 0.09 &&  2245 & 2.06 & 0.25 \\
13 & 10 & 4 && 119 & 0.13 & 0.05 && 267 & 0.64 & 0.22 && 942 & 1.69 & 0.65 &&  3779 & 4.32 & 1.56 \\
13 & 14 & 3 && 92 & 0.05 & 0.02 && 191 & 0.14 & 0.05 && 617 & 0.47 & 0.15 &&  2429 & 1.23 & 0.37 \\
13 & 15 & 3 && 111 & 0.05 & 0.02 && 240 & 0.14 & 0.04 && 806 & 0.50 & 0.14 &&  3246 & 1.35 & 0.36 \\
15 & 11 & 4 && 119 & 0.20 & 0.04 && 251 & 1.15 & 0.14 && 836 & 3.92 & 0.49 &&  3314 & 9.89 & 1.26 \\
15 & 16 & 3 && 100 & 0.05 & 0.02 && 218 & 0.15 & 0.06 && 728 & 0.52 & 0.19 &&  2797 & 1.37 & 0.48 \\
\bottomrule
\end{tabular}
\bigskip
\bigskip

\caption{\small Timings for superelliptic curves $X\colon y^m=f(x)$ when $f\in \Z[x]$ splits into $d$ distinct linear factors. Times are millisecond averages per prime $p\le N$ for a single thread running on a 2.8GHz Cascade Lake Intel CPU. The \texttt{sage} column lists the average time to execute \texttt{CyclicCover(m,f.change\_ring(GF(p)).frobenius\_matrix(1)} in Sage 9.0, the matrix column lists the average time to compute the Cartier--Manin matrix modulo $p$ using algorithm \textsc{ComputeCartierManinMatrices}, and the trace column is the average time to compute the trace of Frobenius via Remark~\ref{rem:traceonly}.}
\label{tab:perfeasy}
\end{table}

\end{document}